\documentclass[12pt, reqno]{amsart}
\usepackage{amsmath, amsthm, amscd, amsfonts, amssymb, graphicx, color}
\usepackage[bookmarksnumbered, colorlinks, plainpages]{hyperref}

\newtheorem{theorem}{Theorem}[section]
\newtheorem{lemma}[theorem]{Lemma}

\newtheorem{corollary}[theorem]{Corollary}
\theoremstyle{definition}
\newtheorem{definition}[theorem]{Definition}
\newtheorem{example}[theorem]{Example}

\theoremstyle{remark}
\newtheorem{remark}[theorem]{Remark}
\numberwithin{equation}{section}

\begin{document}

\setcounter{page}{1}
\title[Short Title]{ New Generalized Geometric Difference Sequence Spaces}

\author[ Saubhagyalaxmi Singh]{Saubhagyalaxmi Singh $^1$}

\address{$^{1}$ Department of Mathematics, Centurion University of Technology and Management, Odisha, India}
\email{ssaubhagyalaxmi@gmail.com, saubhagyalaxmi.singh@cutm.ac.in }




\begin{abstract}
In this current study, the most apparent aspect is to submit a new geometric sequence space. We investigate its topological properties , inclusion relations, Geometric statistical convergence and Geometric property of Orlicz function and . Moreover, we also check its dual properties.

\end{abstract} 
\maketitle

\section{Introduction and preliminaries}
Scientists were able to approach problems in science and engineering from a different angle, non-Newtonian calculus, which was developed between 1967 and 1970. Michael Grossman and Robert Katz explored non-Newtonian calculus, that covers many areas such as geometric, anageometric, and bigeometric calculus in \cite{Grossman}, which, by depending on multiplication rather than addition, is an opposite to the traditional calculus of Newton and Leibnitz. Geometric analysis was named as multiplicative calculus by D.Stanley  \cite{Stanley} and further it was generalised by  D.Campell \cite{Campbell}, Bashirov et al. \cite{Bashirov1} and several other researchers  \cite{Bashirov2, Florack, Aniszewska}.\\

 The traditional difference sequence spaces were first introduced by Kizmaz in 1981 $X\left(\right. \Delta \left. \right)\ ,$ for  $X=c,\, c_{0} ,\, l_{\infty } $ which are the class of convergent, null and bounded sequences respectively, defined as \\
\[X\left(\Delta \right)=\left\{\right. x=\left(x_{i} \right):\left(\Delta x_{i} \right)\in X\left. \right\}\] 
where  $\Delta x_{i} =x_{i} - x_{i+1} .$

\noindent In 1995 Et. and  \c{C}olak \cite{Et3} generalized $X\left(\Delta \right)$ to \textit{m}th order difference sequence spaces by defining $X\left(\right. \Delta ^{m} \left. \right)=\left\{\right. x=\left(x_{i} \right):\left(\right. \Delta ^{m} x_{i} \left. \right)\in X\left. \right\}\ ,$ where $\Delta ^{m} x_{i} =\Delta ^{m} \left(x_{i} - x_{i+1} \right)$.

 T\"{u}rkmen and Ba\c{s}ar \cite{Turkmen}  studied the sets of  geometric real numbers and geometric complex numbers by  $  \mathbb{R}(G)$ and $\mathbb{C}(G)$  respectively as 
 \begin{align*}
&  \mathbb{R}(G)=  \{e^{x}: x\in \mathbb{R} \} = \mathbb{R^{+}} / {0}, \\
 & \mathbb{C}(G)=  \{e^{x}: x\in \mathbb{C} \} = \mathbb{C} / {0}.
 \end{align*}

 Geometric addition $(\oplus)$, geometric subtraction $(\ominus)$ , geometric multiplication $(\odot)$ and geometric limit $_G\lim_{\eta \to \infty}$ defined in \cite{sing2,sing3,sing4, sing5}.

 Infact they introduced the geometric sequence spaces $l_{\infty } \left(G\right),c \left(G\right),c_{0} \left(G\right),l_{p} \left(G\right).$ \cite{sing2,KB1}   

Boruah et al. \cite{KB1, KB2, KB3, KB4,KB5,KB6} have studied various properties of  Geometric sequence spaces . \\ The $\textit{m}^{th}$ order geometric difference sequence spaces as
\[l_{\infty }^{G} \left(\right. \Delta _{G}^{m} \left. \right)=\left\{\right. x=\left(x_{i} \right)\ ;\; \Delta _{G}^{m} x\in l_{\infty }^{G} \left. \right\}\ ,\] 
\[c^{G} \left(\right. \Delta _{G}^{m} \left. \right)=\left\{\right. x=\left(x_{i} \right):\Delta _{G}^{m} x\in c^{G} \left. \right\}\ ,\] 
\[c_{0}^{m} \left(\right. \Delta _{G}^{m} \left. \right)=\left\{\right. x=\left(x_{i} \right):\Delta _{G}^{m} x\in c_{0}^{G} \left. \right\}\ ,\] 
respectively where $m \in \mathbb{N}$
\begin{align*}
& \Delta _{G}^{0}=  (x_i) \\
& \Delta _{G}^{1}x =(\Delta _{G}x_i)= (x_i \ominus x_{i+1}) \\
& \Delta _{G} ^ {2} x= (\Delta _{G} ^{2} x_i)=(\Delta_{G}x_i \ominus \Delta_{G}x_{i+1})\\
& =(x_i \ominus x_{i+1} \ominus  x_{i+1} \oplus  x_{k+2})\\
& =(x_i \ominus e^2 \odot x_{i+1} \oplus x_{i+1})\\
&\Delta _{G} ^ {3} x= (\Delta _{G} ^{3} x_i)=(\Delta_{G}^{2}x_i \ominus \Delta_{G}^{2}x_{i+1})\\
&=(x_i \ominus  e^{2} \odot x_{i+1} \oplus e^{3}\odot  x_{i+1} \ominus  x_{k+3})\\
&\ldots \ldots\ldots\ldots\ldots\ldots \ldots\\
&\Delta _{G} ^ {m} x= (\Delta _{G} ^{m} x_i)=(\Delta_{G}^{m-1}x_i \ominus \Delta_{G}^{m-1}x_{i+1})
\end{align*}
\begin{align} \label{eqem}
& =_G \sum _{\nu =0}^ {m} (\ominus e)^ {\nu_{G}} \odot e^{ \binom {m} {\nu}} \odot x_{k+ \nu}, \textrm {~with~} (\ominus e)^{0_G}=e. 
\end{align}
In 1968  Ces{\`a}ro sequence spaces were brought to lime light by Dutch Mathematical Society  which were of the form\\
$[C,1]=\left\{x=(x_{i})\;\in \;\omega:\displaystyle
\lim_{n \rightarrow
\infty}\frac{1}{n}\sum_{i=1}^{n}|x_{i}-\ell|=0\;\; ,\textrm{~for\;\;
some~}
\ell\right\}.$\;\;\\\
$(C,1)=\left\{x=(x_{i})\;\in \;\omega:\displaystyle
\lim_{n \rightarrow
\infty}\frac{1}{n}\sum_{i=1}^{n}(x_{i}-\ell)=0\;\; \textrm {~,for\;\;
some~}
\ell\right\}.$\;\;

Then after, in 1970, Shiue \cite{shiu1, shiu2} researched the Ces{\`a}ro function and sequence spaces. Since then, other academics have become interested in Ces{\`a}ro sequence spaces.

Leindler \cite{Tzafriri} first developed the strongly$(V,\lambda)-$ summable sequence spaces \\
$$[V,\lambda]=\left\{x=(x_{i})\;\in \;\omega:\displaystyle
\lim_{n \rightarrow
\infty}\frac{1}{\lambda_{n}}\sum_{i\;\in \;I_{n}}|x_{i}-\ell|=0\;\;
\textrm{~,for\;\; some~}\;\; \ell\right\},$$

$$(V,\lambda)=\left\{x=(x_{i})\;\in \;\omega:\displaystyle
\lim_{n \rightarrow
\infty}\frac{1}{\lambda_{n}}\sum_{i\;\in \;I_{n}}(x_{i}-\ell)=0\;\;
\textrm{~,for\;\; some~ ~} \ell\right\},$$

Where, the de la Vallee-Pousin meaning is
$t_{n}(x)=\frac{1}{\lambda_{n}}\displaystyle\sum_{i\;\in \;I_{n}}x_{i}$, 
in which the non-decreasing sequence of positive reals $(\lambda_{n})$  generally tends to infinity with \;$\lambda_{1}=1$
\;and\;$\lambda_{n+1} \leq \lambda_{n}+1$ and  $I_{n}=[n - \lambda_{n}+1,n], n=1,2,3,....$.

The reduction of the $(V,\lambda)-$ summable sequence spaces to the Ces{\`a}ro  sequences space is noticed when $\lambda_{n}=n$.

 The new geometric sequence spaces is now being introduced as follows:
\begin{align*}
 & (C,1)_G (\Delta_{G} ^{m}) =\\ & \left \{ x=(x_i) \in \omega(G) :  {~}_{G}{\mathop{\lim }\limits_{\eta \to \infty }} (e\oslash e^ \eta) \odot{~} _G\sum _{i=1}^{n} (\Delta _{G}^{m} x_{i} \ominus L)^G =1, \textrm{~for some~} L \right\}, \\
& [C,1]_G (\Delta_{G} ^{m}) =\\ & \left \{ x=(x_i) \in \omega(G) : {~}_G {\mathop{\lim }\limits_{\eta \to \infty }} (e\oslash e^ \eta) \odot{~} _G\sum _{i=1}^{n} |\Delta _{G}^{m} x_{i} \ominus L|^G =1 , \textrm{~for some~} L \right\}, \\
& (V, \lambda)_G (\Delta_{G} ^{m}) = \\ & \left \{ x=(x_i) \in \omega(G) : {~}_G {\mathop{\lim }\limits_{\eta \to \infty }} (e\oslash e^{\lambda_\eta}) \odot{~} _G\sum _{i \in I_n} (\Delta _{G}^{m} x_{i} \ominus L)^G =1 , \textrm{~for some~} L \right\}, 
\end{align*}
\begin{align*}
& [V, \lambda]_G (\Delta_{G} ^{m}) =\\ & \left \{ x=(x_i) \in \omega(G) : {~}_G {\mathop{\lim }\limits_{\eta \to \infty }} (e\oslash e^{\lambda_\eta}) \odot{~} _G\sum _{i \in I_n} |\Delta _{G}^{m} x_{i} \ominus L|^G =1 , \textrm{~for some~} L \right\}.
\end{align*}

\section{Main Results}

\subsection{ Topological Structure}

The spaces  $(C,1)_G (\Delta_{G} ^{m}), {~~}[C,1]_G (\Delta_{G} ^{m}) ,  {~~}(V, \lambda)_G (\Delta_{G} ^{m}) $ and  $[V, \lambda]_G (\Delta_{G} ^{m}) $  are Banach spaces have the given  norms
\begin{align}\label{2eq}
& ||x||_{\Delta _{G} }^{G} ={~}_G \sum _{i=1}^{n}\left|\, x_{i} \, \right| ^{G} \oplus {\mathop{\sup }\limits_{\eta \in \mathbb{N}}} \left| (e \oslash e^ \eta ) \odot {~}_G\sum _{i=1}^{n}\Delta _{G}^{m} x_i  \right|^{G} ,\\
& ||x||_{\Delta _{G} }^{G} ={~}_G \sum _{i=1}^{n}\left|\, x_{i} \, \right| ^{G} \oplus {\mathop{\sup }\limits_{\eta \in \mathbb{N}}} \left( (e \oslash e^ \eta ) \odot {~}_G\sum _{i=1}^{n}\left|\Delta _{G}^{m} x_i \right|^{G} \right) ,\\
& ||x||_{\Delta _{G} }^{G} ={~}_G \sum _{i=1}^{n}\left|\, x_{i} \, \right| ^{G} \oplus {\mathop{\sup }\limits_{\eta \in \mathbb{N}}} \left| (e\oslash e^{\lambda_\eta}) \odot {~}_G\sum _{i \in I_n} \Delta _{G}^{m} x_i \right|^{G} \\
& \textrm{~and~}  ||x||_{\Delta _{G} }^{G} ={~}_G \sum _{i=1}^{n}\left|\, x_{i} \, \right| ^{G} \oplus {\mathop{\sup }\limits_{\eta \in \mathbb{N}}} \left( (e\oslash e^{\lambda_\eta}) \odot {~}_G\sum _{i \in I_n}\left|\Delta _{G}^{m} x_i \right|^{G} \right) .
\end{align}
respectively.
\begin{theorem}
The spaces $\left(C,\; 1\right)_{G} \left(\right. \Delta _{G}^{m} \left. \right),\; \left[C,1\right]_{G} \left(\right. \Delta _{G}^{m} \left. \right),\left(V,\lambda \right)_{G} \left(\right. \Delta _{G}^{m} \left. \right)$ and $\left[V,\lambda \right]_{G} \left(\right. \Delta _{G}^{m} \left. \right)$  are linear over $\mathbb{C}(G)$.
\end{theorem}

\begin{proof} 
Let $x=\left(x_{i} \right)\; {\rm and}\; y=\left(y_{i} \right)$ be any two elements of $\left[C,\; 1\right]_{G} \left(\right. \Delta _{G}^{m} \left. \right)$ then there exist $L\, \& \, L'$ s.t.
\[_{G}{\mathop{\lim }\limits_{\eta \to \infty }}(e \oslash e^ \eta ) \odot {~}_ G\sum _{i=1}^{n}\left|{\kern 1pt} {\kern 1pt} \Delta _{G}^{m} x_{i}{\kern 1pt} {\kern 1pt} \ominus  L\right| ^{G} =1,\] 
\[_{G}{\mathop{\lim }\limits_{\eta \to \infty }} (e \oslash e^ \eta ) \odot {~}_G\sum _{i=1}^{n}\left|{\kern 1pt} {\kern 1pt} \Delta _{G}^{m} y_{i}{\kern 1pt} {\kern 1pt} \ominus L'\right| ^{G} =1.\] 
For any two scalars $\alpha ,\beta \in \mathbb{C}(G),$ by setting $L^{''} =\alpha L\oplus \beta L',$ we have
\begin{align*}
& _{G}{\mathop{\lim }\limits_{\eta \to \infty }} (e \oslash e^ \eta ) \odot {~}_ G \sum _{i=1}^{n}\left|\Delta _{G}^{m} \left(\alpha x_{i}\oplus \beta y_{i}\right)\, \ominus L^{''}\right|^{G}  \\ 
& =_{G}{\mathop{\lim }\limits_{\eta \to \infty }} (e \oslash e^ \eta ) \odot {~}_ G\sum _{i=1}^{n}\left|\Delta _{G}^{m} \left(\alpha x_{i}\oplus \beta y_{i}\right){\kern 1pt} {\kern 1pt} {\kern 1pt} \ominus {\kern 1pt} {\kern 1pt} \left(\alpha L\oplus \beta L'\right){\kern 1pt} {\kern 1pt} \right| ^{G}\\ 
& \le \left|\alpha \right| \; _{G}{\mathop{\lim }\limits_{\eta \to \infty }}(e \oslash e^ \eta ) \odot {~}_ G\sum _{i=1}^{n}\left|\, \Delta _{G}^{m} x_{i}\ominus L\, \right| ^{G} \\ & \oplus \left|\, \beta \, \right| \; _{G}{\mathop{\lim }\limits_{\eta \to \infty }} (e \oslash e^ \eta ) \odot {~}_ G\sum _{i=1}^{n}\, \left|\, \Delta _{G}^{m} y_{i}\ominus L'\, \right| ^{G} \\ 
& {\mathop{\to }\limits^{G}}1 \; {\rm as}\; \eta \to \infty \\
& \Rightarrow \alpha x_{i} \oplus \beta y_{i}\in \left[C,\; 1\right]^{G} \left(\right. \Delta _{G}^{m} \left. \right)\ .
\end{align*} 
For other spaces the results  follow using  similar techniques.

\end{proof}

\begin{theorem} \label{thm2.2}
The space $\left(C,\; 1\right)_{G} \left(\right. \Delta _{G}^{m} \left. \right)$  is a normed linear space w.r.t the norm 
$$||x||_{\Delta _{G} }^{G} ={~}_G \sum _{i=1}^{n}\left|\, x_{i} \, \right| ^{G} \oplus {\mathop{\sup }\limits_{\eta \in \mathbb{N}}} \left| (e \oslash e^ \eta ) \odot {~}_G\sum _{i=1}^{n}\Delta _{G}^{m} x_{i}\,  \right|^{G} .$$
\end{theorem} 

\begin{proof}
Consider $x=(x_i),{~~} y=(y_k) \in \left(C,\; 1\right)_{G} \left(\right. \Delta _{G}^{m} \left. \right) ,$ , for normed linear spaces we will prove the following four conditions :
\begin{itemize}
			\item[(i)] $ ||x||_{\Delta _{G} }^{G} \geq 1$
			\item[(ii)] $ ||x||_{\Delta _{G} }^{G}=1 \Leftrightarrow x=0_{G}.$
			\item[(iii)] $||\lambda \odot x||_{\Delta _{G} }^{G} = |\lambda| \odot  ||x||_{\Delta _{G} }^{G},  \textrm{~for~} \lambda \in \mathbb{C}(G) \textrm{~and~} x \in X$,
			\item[(iv)] $|| x \oplus y||_{\Delta _{G} }^{G} \leq ||x||_{\Delta _{G} }^{G} \oplus ||y||_{\Delta _{G} }^{G},  \textrm{~for all~} x,y \in X$, $\textrm {~ (Triangle inequality)~}$
	\end{itemize}
	
\begin{itemize}
\item[(i)]  Here
\begin{align*} & |x_i|^G \geq 1 \textrm{~and~}_G\sum _{i=1}^{n}|\Delta _{G}^{m} x_{i}|^{G} \geq 1.\\
& \textrm{~Consider~} ||x||_{\Delta _{G} }^{G} ={~}_G \sum _{i=1}^{n}\left|\, x_{i} \, \right| ^{G} \oplus {\mathop{\sup }\limits_{\eta \in \mathbb{N}}} \left| (e \oslash e^ \eta ) \odot {~}_G\sum _{i=1}^{n}\Delta _{G}^{m} x_{i}\,  \right|^{G}\\
& = {~}_G \sum _{i=1}^{n}\left|\, x_{i} \, \right| ^{G} . {\mathop{~\sup }\limits_{\eta \in \mathbb{N}}} \left| (e \oslash e^ \eta ) \odot {~}_G\sum _{i=1}^{n}\Delta _{G}^{m} x_{i}\,  \right|^{G} \\
& \geq 1 .
\end{align*} 
\item[(ii)] \begin{align*}
&\left\| x\right\| _{\Delta _{G} }^{G} = 1  \Leftrightarrow {~~}
 _G\sum _{i=1}^{n}\left|x_{i} \right| ^{G} \oplus {\mathop{\sup }\limits_{\eta \in \mathbb{N}}} \left|\, (e \oslash e^ \eta ) \odot {~}_ G\sum _{i=1}^{n}\Delta _{G}^{m} x_{i}\,  \right|^{G} =1\\ 
& \Leftrightarrow \left|x_{1} \right|^{G} .{\mathop{\sup }\limits_{\eta \in \mathbb{N}}} \left|x_{i} \ominus x_{i+1} \right|^ G=1,\; \; \; \forall i\\
&\Leftrightarrow \left|x_{1} \right|^{G} =1\; {\rm and}\; \left|x_{i} \ominus x_{i+1} \right|^{G} =1
\end{align*}
\begin{align*}
& \Leftrightarrow x_{1} =1\; {\rm and}\; x_{i} \ominus x_{i+1} =1,\; \forall i\\ 
&\Leftrightarrow x_{1} =1\; {\rm and}\; x_{i} /x_{i+1} =1,\; \forall i \\
 &\Leftrightarrow x_{1} =1\; {\rm and}\; x_{i} =x_{i+1} ,\; \forall i\\ 
&\Leftrightarrow x_{i} =1,\; \forall i \\
& \Leftrightarrow x = (1, 1, 1,1,........)= 1_G.
\end{align*}
\item[(iii)]    For any scalar $\lambda$, consider
\begin{align*}
& \left\| {\kern 1pt} {\kern 1pt} \lambda \odot x{\kern 1pt} {\kern 1pt} \right\| _{\Delta _{G} }^{G} = {~}_G\sum _{i=1}^{n}\left|\lambda \odot x_{i} \right| ^{G} \oplus {\mathop{\sup }\limits_{m,\, n}} {\kern 1pt} {\kern 1pt} \left|{\kern 1pt} {\kern 1pt} (e \oslash e^ \eta ) \odot {~}_ G\sum _{i=1}^{n}\Delta _{G}^{m} \left(\lambda \odot x\right){\kern 1pt} {\kern 1pt}  \right|^{G} \\ 
& \le \left|\lambda \right| \odot {~}_G\sum _{i=1}^{n}\left|x_{i} \right| ^{G} \oplus \left|\lambda \right| \odot {\mathop{\sup }\limits_{m,\, n}} {\kern 1pt} {\kern 1pt} \left|(e \oslash e^ \eta ) \odot {~}_ G\sum _{i=1}^{n}\Delta _{G}^{m} x_{i}{\kern 1pt} {\kern 1pt}  \right|^{G} \\ 
& =\left|\lambda \right|  \odot \left[_G\sum _{i=1}^{n}\left|x_{i} \right|^{G} \oplus {\mathop{\sup }\limits_{\eta \in \mathbb{N}}} \left|(e \oslash e^ \eta ) \odot {~}_ G\sum _{i=1}^{n}\Delta _{G}^{m} x_{i} {\kern 1pt} {\kern 1pt} \right|^{G}  \right]\\ 
& =\left|\lambda \right| \odot  \left\| x\right\| _{\Delta _{G} }^{G} .
\end{align*}
 \item[(iv)] Suppose 
\begin{equation} 
\left\| x\right\| _{\Delta _{G} }^{G} = _G\sum _{i=1}^{n}\left|x_{1} \right| ^{G} \oplus {\mathop{\sup }\limits_{\eta \in \mathbb{N}}} \left| (e \oslash e^ \eta ) \odot {~}_ G\sum _{i=1}^{n}\Delta _{G}^{m} x_{i}{\kern 1pt} {\kern 1pt}  \right|^{G}
\end{equation} 
and
\begin{equation} 
\left\| y\right\| _{\Delta _{G} }^{G} =_G\sum _{i=1}^{n}\left|y_{i} \right| ^{G} \oplus {\mathop{\sup }\limits_{m,\, n}} \left| (e \oslash e^ \eta ) \odot {~}_ G\sum _{i=1}^{n}\Delta _{G}^{m} x_{i}  \right|^{G}
\end{equation} 

Consider
\begin{align*}
& \left\| x\oplus y\right\| _{\Delta _{G} }^{G} =_G\sum _{i=1}^{n}\left|x_{i} \oplus y_{i} \right| ^{G} \oplus {\mathop{\sup }\limits_{\eta \in \mathbb{N}}} \left|(e \oslash e^ \eta ) \odot {~}_ G\sum _{i=1}^{n}\Delta _{G}^{m} \left(x\oplus y\right) {\kern 1pt} {\kern 1pt} \right|^{G} \\
&  \le \left( _G\sum _{i=1}^{n}\left|x_{i} \right|^G \oplus {\mathop{\sup }\limits_{\eta \in \mathbb{N}}} {\kern 1pt} {\kern 1pt} \left|(e \oslash e^ \eta ) \odot {~}_ G\sum _{i=1}^{n}\Delta _{G}^{m} x_{i}{\kern 1pt} {\kern 1pt}  \right|^{G}\right) \oplus\\ & \left( _G\sum _{i=1}^{n}\left|y_{i} \right| ^{G} \oplus {\mathop{\sup }\limits_{\eta \in \mathbb{N}}} {\kern 1pt} {\kern 1pt} \left|{\kern 1pt} {\kern 1pt} (e \oslash e^ \eta ) \odot {~}_ G\sum _{i=1}^{n}\Delta _{G}^{m} \left(y\right){\kern 1pt} {\kern 1pt}  \right|^{G} \right) \\
&= \left\| x\right\| _{\Delta _{G} }^{G} \oplus \left\| y\right\| _{\Delta _{G} }^{G} .
\end{align*}

\end{itemize}

\end{proof}

\begin{theorem} \label{thm2.3}
The space $\left(C,\; 1\right)_{G} \left(\right. \Delta _{G}^{m} \left. \right)$ is a Banach space  w.r.t. the norm 
$$||x||_{\Delta _{G} }^{G} ={~}_G \sum _{i=1}^{n}\left|\, x_{i} \, \right| ^{G} \oplus {\mathop{\sup }\limits_{\eta \in \mathbb{N}}} \left( (e \oslash e^ \eta ) \odot {~}_G\sum _{i=1}^{n}|\Delta _{G}^{m} x_{i}|^{G}\,  \right) .$$
\end{theorem}

\begin{proof}
 Let $(x^{n}) $ is a Cauchy sequence in $\left(C,\; 1\right)_{G} \left(\right. \Delta _{G}^{m} \left. \right)\ ,$ where $x^{n} =\left(x_{i}^{(n)} \right)= \left(x_{1}^{(n)}, x_{2}^{(n)}, x_{3}^{(n)}, \ldots\ldots \right)$ for $\eta \in \mathbb{N}$ and $x_{j}^{(n)}$ is the $jth$ co-ordinate of $x^{n}$.

\noindent Then
\begin{align*}
& \left\| {\kern 1pt} {\kern 1pt} x^{n} \ominus x^{m} \right\| _{\Delta _{G} }^{G} {\mathop{\to }\limits^{G}} 1\; {\rm as}\; m,\; \eta \to \infty ,\\
& i.e. {~~}_G \sum _{i=1}^{n}{\kern 1pt} {\kern 1pt} \left|x_{i}^{n} \ominus x_{i}^{m} \right| ^{G} \oplus {\mathop{\sup }\limits_{\eta \in \mathbb{N}}} {\kern 1pt} {\kern 1pt} \left|(e \oslash e^ \eta ) \odot {~}_G \sum _{i=1}^{n}\Delta _{G}^{m} \left(\right. x_{i}^{n} \ominus x_{i}^{m} \left. \right) {\kern 1pt} {\kern 1pt} \right|^{G} {\mathop{\to }\limits^{G}} 1\\ 
& i.e. {~}_G \sum _{i=1}^{n}\left|x_{i}^{n} \ominus x_{i}^{m} \right| ^{G} \cdot {\mathop{\sup }\limits_{\eta \in \mathbb{N}}} \left|{\kern 1pt} {\kern 1pt} (e \oslash e^ \eta ) \odot {~}_G \sum _{i=1}^{n}\Delta  _{G}^{m} \left(\right. x_{i}^{n} \ominus x_{i}^{m} \left. \right){\kern 1pt} {\kern 1pt} \right|^{G} {\mathop{\to }\limits^{G}} 1\\
& i.e. {~ ~}_G \sum _{i=1}^{n}\left|{\kern 1pt} {\kern 1pt} x_{i}^{n} \ominus x_{i}^{m} {\kern 1pt} {\kern 1pt} \right| ^{G} {\mathop{\to }\limits^{G}} 1 {\rm ~and~}\; {\mathop{\sup }\limits_{\eta \in \mathbb{N}}} {\kern 1pt} {\kern 1pt} \left|{\kern 1pt} {\kern 1pt} (e \oslash e^ \eta ) \odot {~}_G \sum _{i=1}^{n}\Delta _{G}^{m} \left(\right. x_{i}^{n} \ominus x_{i}^{m} \left. \right) {\kern 1pt} {\kern 1pt} \right|^{G} {\mathop{\to }\limits^{G}} 1.
\end{align*}
 Hence we get $$ \left|x^{n} \ominus x^{m} \right|^{G} {\mathop{\to }\limits^{G}} 1 {~\rm as~} m,\; n{\mathop{\to }\limits^{G}} \infty .$$
Since  $\mathbb{C}\left(G\right)$ is complete,  $(x_{i}^{(n)}) =\left(x_{i}^{(1)}, x_{i}^{(2)}, x_{i}^{(3)}, \ldots\ldots \right)$ is a Cauchy sequence. Hence by  completeness of $\mathbb{C}(G)$, $(x_{i}^{(n)})$ converges to $x_{i}$ .\\ Since $(x^{n}) $ is a Cauchy sequence so for each $\varepsilon>1$, $\exists$  $N=N(\varepsilon)$ s.t.  $ \left\| {\kern 1pt} {\kern 1pt} x^{n} \ominus x^{m} \right\| _{\Delta _{G} }^{G} < \varepsilon$ $\forall$ $m,n \geq N.$\\
Hence we get \\
$_G \sum _{i=1}^{n} \left|x_{i}^{n} \ominus x_{i}^{m} \right| ^{G} < \varepsilon$ and ${\mathop{\sup }\limits_{\eta \in \mathbb{N}}} \left| (e \oslash e^ \eta ) \odot {~}_G \sum _{i=1}^{n}\Delta  _{G}^{m} \left(\right. x_{i}^{n} \ominus x_{i}^{m} \left. \right) \right|^{G} < \varepsilon$ \\ for all $k \in \mathbb{N}$ and $m,n, \geq N. $ So we have
  \begin{align*}
& \left\|  x^{n} \ominus x^{m} \right\| _{\Delta _{G} }^{G}\\
& = _G {\mathop{\lim} \limits_{m \to \infty}}  _G \sum _{i=1}^{n} \left|x_{i}^{n} \ominus x_{i}^{m} \right| ^{G}  \oplus _G {\mathop{\lim} \limits_{m \to \infty}} \mathop{\sup }\limits_{\eta \in \mathbb{N}} \left| (e \oslash e^ \eta ) \odot {~}_G \sum _{i=1}^{n}\Delta  _{G}^{m} \left(\right. x_{i}^{n} \ominus x_{i}^{m}) \right|^{G}\\
& = \varepsilon \oplus \varepsilon \\
 &= \varepsilon^2 .
\end{align*}
Hence the sequence $(x^{n})$ converges of the sequence $x=(x_{i}).$\\
Here only we show that the sequence $x=(x_{i}) \in \left(C,\; 1\right)_{G} \left(\right. \Delta _{G}^{m} \left. \right)\ .$ So we consider 
\begin{align*}
& \left|(e \oslash e^ \eta ) \odot {~}_ G\sum _{i=1}^{n}\Delta _{G}^{m}  x_{i} \right|^{G} \\
 &= \left|(e \oslash e^ \eta ) \odot {~}_ G\sum _{i=1}^{n}\Delta _{G}^{m} \left(\right. x_{i} \oplus x_{i}^{N} \ominus x_{i}^{N} \left. \right)\,  \right|^{G} \\
 & \leq  \left|(e \oslash e^ \eta ) \odot {~}_ G\sum _{i=1}^{n}\Delta _{G}^{m}  x_{i}^{N} \right| ^{G}  \oplus  \left|(e \oslash e^ \eta ) \odot {~}_ G\sum _{i=1}^{n}\Delta _{G}^{m}(x_{i}^{N} \ominus x_{i} \left. \right)\,  \right|^{G} .
\end{align*} 

From the inequalities\\ $_G \sum _{i=1}^{n} \left|x_{i}^{n} \ominus x_{i}^{m} \right| ^{G} < \varepsilon$ and ${\mathop{\sup }\limits_{\eta \in \mathbb{N}}} \left| (e \oslash e^ \eta ) \odot {~}_G \sum _{i=1}^{n}\Delta  _{G}^{m} \left(\right. x_{i}^{n} \ominus x_{i}^{m} \left. \right) \right|^{G} < \varepsilon,$  hence the sequence $(x_{i}^{N}) \in \left(C,\; 1\right)_{G} \left(\right. \Delta _{G}^{m} \left. \right),$ hence 
 $(x_{i}) \in \left(C,\; 1\right)_{G} \left(\right. \Delta _{G}^{m} \left. \right).$ Therefore,  the space $\left(C,\; 1\right)_{G} \left(\right. \Delta _{G}^{m} \left. \right)$ is a Banach space.
\\Hence Proved.
\end{proof}
\begin{remark}
 The spaces $[C,1]_G (\Delta_{G} ^{m}) ,  {~~}(V, \lambda)_G (\Delta_{G} ^{m}) $ and  $[V, \lambda]_G (\Delta_{G} ^{m})$ are Banach spaces with their corresponding norms (\ref{2eq}) using the similar techniques as in  theorem \ref{thm2.3}.
\end{remark}
\begin{theorem}
\begin{itemize}
\item[(i)] $ [C,\; 1]_{G} (\Delta _{G}^{m}) \subset (C,1)_{G} ( \Delta _{G}^{m})$
 \item[(ii)] $ \left[V,\lambda \right]_{G} \left(\right. \Delta _{G}^{m} \left. \right)\subset \left(V,\; \lambda \right)_{G} \left(\right. \Delta _{G}^{m} \left. \right)$ 
 \end{itemize} and the inclusions are strict.
\end{theorem}

\begin{proof}
 (i) Suppose $x\in \left[C,\; 1\right]_{G} \left(\right. \Delta _{G}^{m} \left. \right)$  then there exists a $L\in \mathbb{C}\left(G\right)$ s.t.
 \begin{align*}
& _{G}{\mathop{\lim }\limits_{\eta \to \infty }} (e \oslash e^ \eta ) \odot {~}_G \sum _{i=1}^{n}\left|\Delta _{G}^{m} \left(x_{i} \right){\kern 1pt} {\kern 1pt} \ominus L\right| ^{G} =1\\
& \Rightarrow _{G}{\mathop{\lim }\limits_{\eta \to \infty }} (e \oslash e^ \eta ) \odot {~}_G \sum _{i=1}^{n}\left(\right. \Delta _{G}^{m} \left(x_{i} \right){\kern 1pt} {\kern 1pt} \ominus L\left. \right) ^{G} =1.
\end{align*} 
So it means $x\in \left(C,1\right)_{G} \left(\right. \Delta _{G}^{m} \left. \right)\ .$\\
Hence $ [C,\; 1]_{G} (\Delta _{G}^{m}) \subset (C,1)_{G} ( \Delta _{G}^{m})$.\\
Similarly we can prove (ii). 
\end{proof}


\noindent 
\section{  Geometric Statistical Convergence}

\noindent In this section, we define a $\lambda-$ geometric statistical convergence $S_{\lambda } \left(\right. \Delta _{G}^{m})$ and  then establish the  relationship of $S_{\lambda } \left(\right. \Delta _{G}^{m} \left. \right)$ with $\left[V,\; \lambda \right]_{G} \left(\right. \Delta _{G}^{m} \left. \right)$ and $\left[C,1 \right]_{G} \left(\right. \Delta _{G}^{m} \left. \right) .$ The idea of statistical convergence was introduced by Fast \cite{Fast} and studied subsequently in different spaces by several  authors including \cite{Connor, Et6, Et7} .




\par Now we define  geometric statistical convergence separately for  the spaces $[C,1]_{G}\left(\right. \Delta _{G}^{m})$ and   $[V, \lambda]_{G}\left(\right. \Delta _{G}^{m})$.
\begin{definition}
{\rm Now we define 
  a sequence $x=(x_{i})$ is said to be $\Delta^{m}-$ geometric statistically convergent $\Delta^{m}_{G}-$ statistically convergent if there is a complex number $L$ such that 
$$_{G} \lim _{\eta \to \infty} (e \oslash e^{\eta}) \odot  |\{ k \leq n: | \Delta ^{m}_{G} x_{i} \ominus L| \geq \varepsilon \}|=1$$
for every  $\varepsilon >1$. In this case we write $x_{i}  {\mathop {\to} \limits ^{G}} LS(\Delta ^{m}_{G})$. The set $\Delta ^{m}_{G}-$ statistically convergent sequences will be denoted by $S(\Delta^{m}_{G})$.}
\end{definition}
\begin{definition}
 {\rm  A sequence $x=\left(\right. x_{i} \left. \right)$ is said to be $\lambda  {\kern 1pt} {\rm -}$ geometric statistically convergent or $S_{\lambda  } {\kern 1pt} {\rm -}$ geometric convergent to \textit{L}, if for every $\varepsilon >1, $
\[_{G}{\mathop{\lim }\limits_{\eta \to \infty }} (e \oslash e^{\lambda_{\eta }}) \odot  \left|\left\{k\in I_{n} :\left|\Delta _{G}^{m} x_{i} \, \ominus \, L\right|^{G} \ge \varepsilon \right\}\, \right|^{G} =1.\] 
We write

\noindent $S_{\lambda  } \ominus \, _G \lim x=L$ or $ x_{i} {~~} {\mathop{\to }\limits^{G}} L\left(\right. S_{\lambda }( \Delta _{G}^{m}) \left. \right)$

\noindent that is  $S_{\lambda}(\Delta _{G}^m) =\left\{\right. x\in \omega \left(G\right):S_{\lambda}  \ominus   _G\lim x=L{\rm ~for}\; {\rm some}\, L\left. \right\}.$}
\end{definition}
\begin{theorem}
Let $\lambda =\left(\lambda _{k} \right)$ be a non-decreasing sequence  of
positive reals which tend to infinity with \;$\lambda_{1}=1$
\;and\;$\lambda_{k+1} \leq \lambda_{k}+1$ and  $I_{k}=[k - \lambda_{k}+1,k], k=1,2,3,....$.. .\\ Then
\begin{item}
\item[(i)]  If $x_{i}  {\mathop{\to }\limits^{G}}  L\left[V,\; \lambda \right]_{G} \left(\right. \Delta _{G}^{m} \left. \right) ,\; then\; x_{i}  {\mathop{\to }\limits^{G}} L (S_{\lambda }(\Delta _{G}^{m} )) $.

\item[(ii)] If $x\in l_{\infty }^{G} \left(\right. \Delta _{G}^{m} \left. \right)\; and\; x_{i}{\mathop{\to }\limits^{G}}   L\left(\right. S_{\lambda }( \Delta _{G}^{m}) \left. \right),\; then\; x_{i} {\mathop{\to }\limits^{G}}  L{\kern 1pt} {\kern 1pt} \left[V,\; \lambda \right]_{G} \left(\right. \Delta _{G}^{m} \left. \right)\ $ and hence $ x_{i} {\mathop{\to }\limits^{G}}  L{\kern 1pt} {\kern 1pt} \left[C,1 \right]_{G} \left(\right. \Delta _{G}^{m} \left. \right)\ $ provided $x= (x_i)$ is not eventually constant.
\item[(iii)] $S_{\lambda }\left(\right. \Delta _{G}^{m} \left. \right)\ \bigcap l_{\infty }^{G} \left(\right. \Delta _{G}^{m} \left. \right)=\left[V,\; \lambda \right]_{G} \left(\right. \Delta _{G}^{m} \left. \right)\bigcap l_{\infty }^{G} \left(\right. \Delta _{G}^{m} \left. \right) ,$

where $l_{\infty }^{G} \left(\right. \Delta _{G}^{m} \left. \right)=\left\{x=(x_{i}):\Delta _{G}^{m}x \in l_{\infty }^{G} \right\}\ $.
\end{item}
\end{theorem}

\begin{proof}
(i) Suppose $\varepsilon >1$ and $ x_{i}{\mathop{\to }\limits^{G}} L\left[V,\; \lambda \right]_{G} \left(\right. \Delta _{G}^{m} \left. \right) ,$ then we have 
\begin{align*}
& {~}_G \sum _{k\in I_{n} }\left|\Delta _{G}^{m} x_{i}\ominus L\right|^{G}  \ge {~}_G \sum _{{\mathop{k\in l_{n} }\limits_{\left|\Delta _{G}^{m} x_{i}\ominus L\right|\ge \in }} }\left|\Delta _{G}^{m}x_{i}\ominus L\right|^{G}  \\ 
& \ge \varepsilon \left|\, \left\{k\in I_{n} :\left|\Delta _{G}^{m} x_{i}\ominus L\right|^{G} \ge \varepsilon \right\}\, \right|^{G} .
\end{align*}
\par Therefore $x_{i} {\mathop{\to }\limits^{G}}  L\left[V,\; \lambda \right]_{G} \left(\right. \Delta _{G}^{m} \left. \right)\Rightarrow x_{i} {\mathop{\to }\limits^{G}} LS_\lambda \left(\right. \Delta _{G}^{m} \left. \right) .$\\
(ii) Suppose $x_{i} {\mathop {\to} \limits ^{G}} LS_{\lambda}(\Delta_{G}^{m})$ and  $x\in l_{\infty } \left(\right. \Delta _{G}^{m} \left. \right)\; $   
\\$\left|\Delta _{G}^{m} x_{i} \ominus L\right|^{G} \le M$  
for all $k.$ Given $\varepsilon >1$, we have 

\begin{align*}
&  (e \oslash e^{\lambda_{\eta }}) \odot {~}_G \sum _{k\in l_{n} }\left|\, \Delta _{G}^{m} x_{i} \ominus L\, \right| ^{G}\\
& =  (e \oslash e^{\lambda_{\eta }}) \odot {~} _G \sum _{{\mathop{k\in I_{n} }\limits_{\left|\, \Delta _{G}^{m} x_{i}\ominus L\, \right|\, \ge \, \varepsilon }} }\left|\, \Delta _{G}^{m} \left(x_{i} \right)\ominus L\, \right|^{G}  \\
&\oplus  (e \oslash e^{\lambda_{\eta }}) \odot {~}_G \sum _{{\mathop{k\in I_{n} }\limits_{\left|\Delta _{G}^{m} x_{i}\ominus L \right| < \varepsilon }}} \left| \Delta _{G}^{m} x_{i}   \ominus L \right|^{G}\\
& \leq  M(e \oslash e^{\lambda_{\eta }}) \odot \left| \left\{ k\in I_{n} :\left| \Delta _{G}^{m} x_{i}\ominus L\right|^{G} \ge \varepsilon \right \}\right |^{G} \oplus 
\varepsilon .  
\end{align*}
As $n \to \infty ,$ the right hand side  goes to one, this gives $x_{i} \to L\left[V,\; \lambda \right]_{G} \left(\right. \Delta _{G}^{m} \left. \right)\ .$\\
If $x_{i} \in L [V, \lambda]^{G} \left(\right. \Delta _{G}^{m} \left. \right)$\\  then to show  $x_{i} \in L (C,1)^{G} \left(\right. \Delta _{G}^{m} \left. \right).$\\
Consider
\begin{align*}
& x_{i} \in L [V, \lambda]^{G} ( \Delta _{G}^{m})\\
& \Rightarrow(e \oslash e^{\lambda_{\eta }}) \odot {~}_G \sum _{k\in l_{n} }\left|\, \Delta _{G}^{m} x_{i} \ominus L\, \right| ^{G}  \leq \varepsilon , \textrm{~for some~} L .
\end{align*}
To show $$(e \oslash e^{\eta}) \odot {~}_G \sum _{k=1 }^{n}\left(\, \Delta _{G}^{m} x_{i} \ominus L\, \right) ^{G}  \leq \varepsilon , \textrm{~for some~} L .$$
Consider \begin{align*}
& (e \oslash e^{\eta}) \odot {~}_G \sum _{k=1 }^{n}\left(\, \Delta _{G}^{m} x_{i} \ominus L\, \right) ^{G}\\
& \Rightarrow (e \oslash e^{\eta}) \odot {~}_G \sum _{k=1 }^{n-\lambda_{n}}\left(\, \Delta _{G}^{m} x_{i} \ominus L  \, \right) \oplus (e \oslash e^{\eta}) \odot {~}_G \sum _{k=n-\lambda_{n}+1 }^{n}\left(\, \Delta _{G}^{m} x_{i} \ominus L  \, \right) ^{G}\\
& \leq \left( (e \oslash e^{\lambda_{\eta }}) \odot {~}_G \sum _{k\in l_{n}}\left| \Delta _{G}^{m} x_{i} \ominus L  \, \right|^{G} \right) \oplus \left((e \oslash e^{\lambda_{\eta }}) \odot {~}_G \sum _{k\in l_{n}}^{n}\left|\, \Delta _{G}^{m} x_{i} \ominus L  \, \right| ^{G}\right)\\
&  \leq  e^{2} \odot (e \oslash e^{\lambda_{\eta }}) \odot {~}_G \sum _{k\in l_{n}}\left|\Delta _{G}^{m} x_{i} \ominus L  \, \right|^{G} 
\end{align*}
we obtain $x_{i} \rightarrow L(C,1)_{G}(\Delta^{m}_{G}).$\\
(iii) From (i) and (ii) it is obvious.
\end{proof}

\begin{theorem}
If $_G {\mathop{\lim }\limits_{\eta \to \infty}}  \inf (e^{\lambda _{\eta}} \ominus e^{\eta}) >1,$ then $S (\Delta _{m}^{G}) \subset S_{\lambda} (\Delta _{m}^{G}) .$
\end{theorem}
\begin{proof}
 Given $\varepsilon >1,$ we have
 \begin{align*}
&\left\{k\in I_{n} :\left|\, \Delta _{G}^{m} x_{i}{\kern 1pt} {\kern 1pt} \ominus L\, \right|^{G} \ge \varepsilon \right\}\subset \left\{k\le n:\left|\, \Delta _{G}^{m} x_{i}{\kern 1pt} {\kern 1pt} \ominus L\right|^{G} \ge \, \varepsilon \right\}.
\end{align*}
Therefore,
\begin{align*}
& (e \oslash e^{\eta}) \left|\left\{k\le n:\left|\Delta _{G}^{m} x_{i}{\kern 1pt} {\kern 1pt} \ominus  L\right|^{G} \ge \varepsilon \right\}\right|^{G} \\
& \ge  (e \oslash e^{\eta}) \left|\left\{k\in I_{n} :\left|\Delta _{G}^{m} x_{i}{\kern 1pt} {\kern 1pt} \ominus  L\right|^{G} \ge \varepsilon \right\}\right|^{G} \\ 
&=(e^{\lambda_{n}} \oslash e^{\eta}) \odot  (e \oslash e^{\lambda_{\eta }}) \left|\left\{k\in I_{n} :\left|\Delta _{G}^{m} x_{i}{\kern 1pt} {\kern 1pt} \ominus L\right|^{G} \ge \varepsilon \right\}\right|^{G} .
\end{align*} 
Taking the limit as $ n \to \infty $  and by hypothesis $_G {\mathop{\lim }\limits_{\eta \to \infty}}  \inf (e^{\lambda _{\eta}} \ominus e^{\eta}) >1,$ we get $x_{i}{\mathop{\to }\limits^{G}} L S (\Delta _{m}^{G}) \Rightarrow x_{i}{\mathop{\to }\limits^{G}}  L S_{\lambda} (\Delta _{m}^{G}) .$
\end{proof}

\section{ New Geometric sequence spaces defined by Orlicz functions}

 This part  deals with some new types of geometric difference sequence spaces which includes  Orlicz functions. Also here we  examine their topological properties.\\
   Lindenstrauss and Tzafriri \cite{Tzafriri} constructed the sequence space $l_M$ by using the Orlicz function $M$  as follows:
  $$ l_M =\left \{ x=(x_i) \in \omega : \sum_{k=1} ^ \infty M \left( \frac{| x_i|}{\rho}\right) < \infty, \textrm{~ for some~} \rho >0 \right \}.
$$  
With norm $$ \lVert x  \rVert =  inf  \left \{ \rho >0 : {\sum_{k=1}^\infty}  M \left( \frac{| x_i|}{\rho}\right) \leq 1 \right \}$$ this is a Banach space. \\

For  $M$   an Orlicz function, $m$  a positive integer  also $p\left(G\right)=\left(\right. p_{k} \left. \right)_{G} $  any sequence of strictly positive real numbers now we  define the spaces,
\begin{align*}
 & [V,\; \lambda, M, p]^{G} (\Delta_{G}^{m})  = \\
 & \bigg\{ x= \left(x_{i} \right):{~}_{G}{\mathop{\lim }\limits_{\eta \to \infty }}  (e \oslash e^{\lambda_{\eta }}) \odot {~}_G \sum _{k\in I_{n} }\left[M \odot  \left(\frac{\left|\Delta _{G}^{m} x_{i}\ominus L\, \right|^{G} }{\rho } G \right)\right]^{\left(p_{k} \right)^{G} }\\ &=1,
 \textrm{~ for some~} L \textrm{~and~} \rho>1 \bigg\},\\
& [V,\; \lambda, M, p]^{G}_{0} (\Delta_{G}^{m})  = \bigg\{ x= \left(x_{i} \right):_{G}{\mathop{\lim }\limits_{\eta \to \infty }}  (e \oslash e^{\lambda_{\eta }}) {~}_G \sum _{k\in I_{n} }\left[M \odot  \left(\frac{\left|\Delta _{G}^{m} x_{i} \right|^{G} }{\rho } G \right)\right]^{\left(p_{k} \right)^{G} }\\ &=1,
 \textrm{~ for some~}  \rho>1 \bigg\},\\
 &[V,\; \lambda, M, p]^{G}_ {\infty} (\Delta_{G}^{m})  = \bigg\{ x= \left(x_{i} \right):{\mathop{\sup }\limits_{n\in \mathbb{N} }}  (e \oslash e^{\lambda_{\eta }}) {~}_G \sum _{k\in I_{n} }\left[M \odot  \left(\frac{\left|\Delta _{G}^{m} x_{i} \right|^{G} }{\rho } G \right)\right]^{\left(p_{k} \right)^{G} }\\  & < \infty,
 \textrm{~ for some~}  \rho>1 \bigg\}.
\end{align*}
\noindent For $p\left(G\right)=\left(\right. p_{k} \left. \right)_{G} = 1, \forall i $, we denote the above spaces by \\ $[V,\; \lambda, M]^{G} (\Delta_{G}^{m}),  [V,\; \lambda, M]^{G}_{0} (\Delta_{G}^{m})$  and $[V,\; \lambda, M]^{G}_ {\infty} (\Delta_{G}^{m})$  respectively. \\Then we have certain topological properties of the spaces.

\begin{theorem}\label{thm ch 2 2.4.1}
For a bounded sequence $p^{G}=\left(p_{k} \right)^{G}$ of strictly positive geometric real numbers, the space $[V,\; \lambda, M, p]^{G} (\Delta_{G}^{m}),  [V,\; \lambda, M, p]^{G}_{0} (\Delta_{G}^{m})$  and $[V,\; \lambda, M, p]^{G}_ {\infty} (\Delta_{G}^{m})$  are linear over $\mathbb{C}(G),$ the field of geometric complex numbers.
\end{theorem}

\begin{proof}
Let $x,\; y\in [V,\; \lambda, M, p]^{G}_{0} (\Delta_{G}^{m})$ and $\alpha ,\; \beta \in \mathbb{C}(G) .$ Then there exist positive numbers $\rho _{1} ,\; \rho _{2} $ such that
\[_{G}{\mathop{\lim }\limits_{\eta \to \infty }}  (e \oslash e^{\lambda_{\eta }}) \odot {~}_G \sum _{k\in 1_{n} }\left[M \odot \left(\frac{\left|\, \Delta _{m}^{G} x_{i}\, \right|^{G} }{\rho _{1} } G \right)\right] ^{\left(p_{k} \right)^{G} } =1,\] and
\[ _{G}{\mathop{\lim }\limits_{\eta \to \infty }}  (e \oslash e^{\lambda_{\eta }}) \odot {~}_G \sum _{k\in 1_{n} }\left[M \odot \left(\frac{\left|\, \Delta _{m}^{G} y_{i}\, \right|^{G} }{\rho _{2} }G \right)\right] ^{\left(p_{k} \right)^{G} } =1.\] 
Define $\rho _{3} =\max \left\{\right. 2\left|\, \alpha \, \right|\, \rho _{1} ,\; 2\left|\, \alpha \, \right|\, \rho _{2} \left. \right\} .$

\noindent Since $\Delta _{m}^{G} $ is linear and \textit{M} is non-decreasing 
\begin{align*}
&  (e \oslash e^{\lambda_{\eta }}) \odot {~}_G \sum _{k\in 1_{n} }\left[M \odot \left(\frac{\left|\, \Delta _{m}^{G} \left(\alpha x_{i}\oplus \beta y_{i}\right)\, \right|^{G} }{\rho _{3} }G \right)\right] ^{\left(p_{k} \right)^{G} }\\ 
& = (e \oslash e^{\lambda_{\eta }}) {~}\odot _G \sum _{k\in 1_{n} }\left[M \odot \left(\frac{\left|\, \alpha \Delta _{m}^{G} x_{i}\oplus \beta \Delta _{m}^{G} y_{i}\, \right|^{G} }{\rho _{3} } G\right)\right] ^{\left(p_{k} \right)^{G} } \\
& \le  (e \oslash e^{\lambda_{\eta }}) {~}\odot _G \sum _{k\in 1_{n} }\left[M \odot \left(\left(\frac{\left|\, \alpha\Delta _{m}^{G} x_{i}\, \right|^{G} }{\rho _{3} }G \right)\oplus  \left(\frac{\left|\, \beta \Delta _{m}^{G} y_{i}\, \right|^{G} }{\rho _{3} } G\right)\right)\right]^{\left(p_{k} \right)^{G} } \\
&\le  (e \oslash e^{\lambda_{\eta }}) \odot {~} _G \sum _{k\in 1_{n} }\frac{1}{2^{({p^{k})_{G} }} }\left[M \odot \left(\frac{\left|\Delta _{m}^{G} x_{i}\right|^{G} }{\rho _{1} }G \right)\oplus M \odot \left(\frac{\left|\Delta _{m}^{G} y_{i}\right|^{G} }{\rho _{2} }G \right)\right]^{\left(p_{k} \right)^{G} }  \\ 
&\le  (e \oslash e^{\lambda_{\eta }}) \odot {~} _G \sum _{k\in 1_{n} }\left[M \odot \left(\frac{\left|\Delta _{m}^{G} x_{i}\right|^{G} }{\rho _{1} } G\right)\oplus M \odot \left(\frac{\left|\Delta _{m}^{G} y_{i}\right|^{G} }{\rho _{2} }G \right)\right]^{\left(p_{k} \right)^{G} }  \\ 
&\le C (e \oslash e^{\lambda_{\eta }}) \odot {~}_G \sum _{k\in 1_{n} }\left[M\odot\left(\frac{\left|\Delta _{m}^{G} x_{i}\right|^{G} }{\rho _{1} }G \right)\right] ^{\left(p_{k} \right)^{G} } \\ & \oplus  C \odot  (e \oslash e^{\lambda_{\eta }}) {~}_G \sum _{k\in 1_{n}}\left[M \odot \left(\frac{\left|\Delta _{m}^{G} y_{i}\right|^{G} }{\rho _{2} }G \right)\right]^{\left(p_{k} \right)^{G} }  \\ 
& {\mathop{\to }\limits^{G}} 1,
\end{align*} 
where $C=\max \left\{\right. 1,\; 2^{H-1} \left. \right\},$ $H={\mathop{\sup }\limits_{i \in \mathbb{N}}} \left(p_{k} \right)^{G}  .$

\noindent So that $\alpha x_{i}\oplus \beta y_{i}\in \left[V,\; \lambda, M, p \right]_{0}^{G} (\Delta _{G}^{m} ).$

\noindent Hence $ \left[V,\; \lambda, M, p \right]_{0}^{G} (\Delta _{G}^{m} )$ is a linear space.

Other spaces can be proved by using similar techniques.
\end{proof}
\begin{theorem} \label{thm2.4.2}
A bounded sequence $p^{G}=(p_{k})^{G}$ of strictly positive real numbers, $ \left[V,\; \lambda, M, p \right]_{0}^{G} (\Delta _{G}^{m} )$ is a paranormed space (not necessarily totally paranormed) with a paranorm
\begin{align*}
&g(x)=\\
& \inf \left\lbrace \rho^ \frac{(p_{n})^{G}}{H} :  \left( (e \oslash e^{\lambda_{\eta }}) \odot _G \sum _{k\in I_{n} }\left[M \odot  \left(\frac{\left|\Delta _{G}^{m} x_{i} \right|^{G} }{\rho } G \right)\right]^{\left(p_{k} \right)^{G} } \right)^\frac{1}{H} \leq 1, n=1,2,3,....\right\rbrace
\end{align*}
where $H= max \{1, {\mathop{\sup }\limits_{i \in \mathbb{N} }p_{k}^{G}}\}$. For  $m$ be a positive integer,  $M$  any Orlicz function. 
\end{theorem}
\begin{proof}
It is obvious that  $g(x)= g(\ominus x).$ The subadditivity of $g$ follows from the proof of  theorem (\ref{thm ch 2 2.4.1}) taking $\alpha =1 , $ $\beta =1$. It is trivial that $\Delta ^{m}_{G} x=1$ for $x=1$. Since $M(1)=1$, we get $ \inf { \rho^ \frac{p_{n}}{H}}=1$ for $x=1$.\\
Now, here we have to prove scalar multiplication is continuous. Let us take $\alpha$ be any complex number . For the continuity of scalar multiplication  let $\alpha$ be fixed and  $x \rightarrow 1$ in $ \left[V,\; \lambda, M, p \right]_{0}^{G} (\Delta _{G}^{m} )$\\
Consider
\begin{align*}
& g(\alpha  \odot x)= \inf  \bigg\{ \rho^ \frac{(p_{n})^{G}}{H} :  \left( (e \oslash e^{\lambda_{\eta }}) \odot _G \sum _{k\in I_{n} }\left[M \odot  \left(\frac{\left|\Delta _{G}^{m}(\alpha x_{i}) \right|^{G} }{\rho } G \right)\right]^{\left(p_{k} \right)^{G} } \right)^\frac{1}{H} \leq 1, \\ & n=1,2,3,..... \bigg\}\\
& = \inf  \bigg\{ \rho^ \frac{(p_{n})^{G}}{H} :  \left( (e \oslash e^{\lambda_{\eta }}) \odot _G \sum _{k\in I_{n} }\left[M \odot  \left(\frac{\left| \alpha\Delta _{G}^{m}( x_{i}) \right|^{G} }{\rho } G \right)\right]^{\left(p_{k} \right)^{G} } \right)^\frac{1}{H} \leq 1, \\ &  n=1,2,3,......
 \bigg\} \\
& (\because \textrm{~the linearity of~}  \Delta_{m}^{G}).
\end{align*}
Then 
\begin{align*}
& g(\alpha  \odot x)
= \inf \bigg\{ ( |\alpha|r)^ \frac{(p_{n})^{G}}{H} :  \left( (e \oslash e^{\lambda_{\eta }}) \odot _G \sum _{k\in I_{n} }\left[M \odot  \left(\frac{\left|\Delta _{G}^{m} x_{i} \right|^{G} }{r } G \right)\right]^{\left(p_{k} \right)^{G} } \right)^\frac{1}{H} \leq 1, \\ & n=1,2,3,.... \bigg\}
\end{align*}
where $r= \rho / |\alpha|$. Since $ |\alpha|^ {(p_{n})^{G}} \leq  max\{1,  |\alpha|^{\sup {(p_{n})^{G}}\}}$ we have \\
$g(\alpha  \odot x) \leq max\{1,  |\alpha|^{\sup {(p_{n})^{G}}}\}^ \frac{1}{H} \odot $
\begin{align} \label{2.4.1}
& \inf  \bigg\{ ( r^ \frac{(p_{n})^{G}}{H} :  \left( (e \oslash e^{\lambda_{\eta }}) \odot _G \sum _{k\in I_{n} }\left[M \odot  \left(\frac{\left|\Delta _{G}^{m} x_{i} \right|^{G} }{r } G \right)\right]^{\left(p_{k} \right)^{G} } \right)^\frac{1}{H} 
\leq 1, n=1,2,3,.... \bigg\},
\end{align}
Therefore $\alpha \odot x \rightarrow 1$ in $ \left[V,\; \lambda, M, p \right]_{0}^{G} (\Delta _{G}^{m} )$. Let $x$ be fixed and $\alpha \rightarrow 1$, so equation (\ref{2.4.1})  converges to one .\\
Let $\alpha_{i} \rightarrow 1$ as $i\rightarrow \infty$. Let $x$ be a fixed sequence in $ \left[V,\; \lambda, M, p \right]_{0}^{G} (\Delta _{G}^{m} )$. For arbitrary $\varepsilon >1$, let $N$ be a positive integer such that 
\begin{align*}
 (e \oslash e^{\lambda_{\eta }}) \odot _G \sum _{k\in I_{n} }\left[M \odot  \left(\frac{\left|\Delta _{G}^{m} x_{i} \right|^{G} }{\rho } G \right)\right]^{\left(p_{k} \right)^{G} } 
\leq \left( \frac{\varepsilon}{2}G \right)^{H}.
\end{align*}
This implies
\begin{align*}
\left( (e \oslash e^{\lambda_{\eta }}) \odot _G \sum _{k\in I_{n} }\left[M \odot  \left(\frac{\left|\Delta _{G}^{m} x_{i} \right|^{G} }{\rho } G \right)\right]^{\left(p_{k} \right)^{G} } \right)^\frac{1}{H}
\leq  \frac{\varepsilon}{2}G.
\end{align*}
for some $\rho >1$ and all $n>N$.
Using convexity of $M$, let $1<|\alpha| <e$,   we get\\
\begin{align*}
& (e \oslash e^{\lambda_{\eta }}) \odot _G \sum _{k\in I_{n} }\left[M \odot  \left(\frac{\left|\alpha \Delta _{G}^{m} x_{i} \right|^{G} }{\rho } G \right)\right]^{\left(p_{k} \right)^{G} } \\
& \leq  (e \oslash e^{\lambda_{\eta }}) \odot _G \sum _{k\in I_{n} }\left[|\alpha| \odot M \odot  \left(\frac{\left| \Delta _{G}^{m} x_{i} \right|^{G} }{\rho } G \right)\right]^{\left(p_{k} \right)^{G} } \\
&\leq  \left(\frac{\varepsilon}{2}G\right)^H.
\end{align*}
Since $\alpha \rightarrow 1$, corresponding to $\varepsilon > 1$ chosen earlier, we can take one $\delta >1$ depending upon $\varepsilon$ s.t. $|\alpha| < \delta$, which implies
\begin{align*}
&  \left((e \oslash e^{\lambda_{\eta }}) \odot _G \sum _{k\in I_{n} }\left[M \odot  \left(\frac{\left|s_{i} \Delta _{G}^{m} x_{i} \right|^{G} }{\rho } G \right)\right]^{\left(p_{k} \right)^{G} }\right) ^\frac{1}{H} \leq \frac{\varepsilon}{2}G, \textrm{~for~} n \leq N.
\end{align*}
Hence  
\begin{align*}
&  \left((e \oslash e^{\lambda_{\eta }}) \odot _G \sum _{k\in I_{n} }\left[M \odot  \left(\frac{\left|s_{i} \Delta _{G}^{m} x_{i} \right|^{G} }{\rho } G \right)\right]^{\left(p_{k} \right)^{G} }\right) ^\frac{1}{H} \\
& \leq  \left( \left((e \oslash e^{\lambda_{\eta }}) \odot _G \sum _{n \leq N }\left[M \odot  \left(\frac{\left|s_{i} \Delta _{G}^{m} x_{i} \right|^{G} }{\rho } G \right)\right]^{\left(p_{k} \right)^{G} }\right) ^\frac{1}{H} \right) \oplus \\
& \left( \left((e \oslash e^{\lambda_{\eta }}) \odot _G \sum _{n > N }\left[M \odot  \left(\frac{\left|s_{i} \Delta _{G}^{m} x_{i} \right|^{G} }{\rho } G \right)\right]^{\left(p_{k} \right)^{G} }\right) ^\frac{1}{H} \right)\\
& < \frac{\varepsilon}{2} \oplus \frac{\varepsilon}{2} \\ 
& = \varepsilon
\end{align*}
for $|\alpha| < min(1, \delta)$. Therefore $\alpha \odot x \rightarrow 1$ in $\left[V,\; \lambda, M, p \right]_{0}^{G} (\Delta _{G}^{m} )$.
\end{proof}
\begin{theorem} \label{thm2.4.3}
Let $X$ stand for  the spaces $[V, \lambda,M]^{G}, [V, \lambda,M]^{G}_{0}$ or $[V, \lambda,M]^{G}_{\infty}$ and $m \geq e$, the geometric identity.  The inclusion $X(\Delta ^{m-1}) \subset X(\Delta^{m})$ is strict   in general $X(\Delta ^{i}) \subset X(\Delta ^{m})$ for all $i= 1,2,....m-1$.
\end{theorem}
\begin{proof}
We give only  proof for the space $X= [V, \lambda, M]_\infty^{G}$ , and the proof for other geometric spaces follows in a similar way . \\
Let $x \in  [V, \lambda, M]_\infty^{G}(\Delta^{m-1}_{G})$. Then we have 
\begin{align}\label{ch2th2.4.3}
& \sup _{\eta \in \mathbb{N}} (e \oslash e^{\lambda_{\eta }}) \odot _G \sum _{k\in I_{n} }\left[M \odot  \left(\frac{\left|\Delta _{G}^{m-1} x_{i} \right|^{G} }{\rho } G \right)\right] < \infty,
\end{align}
for some $\rho >1$.\\
As $M$ is convex function and non-decreasing  , we have
\begin{align*}
& (e \oslash e^{\lambda_{\eta }}) \odot _G \sum _{k\in I_{n} }\left[M \odot  \left(\frac{\left|\Delta _{G}^{m} x_{i} \right|^{G} }{2\rho } G \right)\right] \\
& = (e \oslash e^{\lambda_{\eta }}) \odot _G \sum _{k\in I_{n} }\left[M \odot  \left(\frac{\left|\Delta _{G}^{m-1} x_{i} \ominus \Delta _{G}^{m-1} x_{i+1} \right|^{G} }{2\rho } G \right)\right]\\
&  = (e \oslash e^{\lambda_{\eta }}) \odot _G \sum _{k\in I_{n} }\left[ \frac{1}{2}_{G} M \odot  \left(\frac{\left|\Delta _{G}^{m-1} x_{i}  \right|^{G} }{\rho } G \right)\right] \oplus \\ & (e \oslash e^{\lambda_{\eta }}) \odot _G \sum _{k\in I_{n} }\left[ \frac{1}{2}_{G} M \odot  \left(\frac{\left|\Delta _{G}^{m-1} x_{i+1}  \right|^{G} }{\rho } G \right)\right] < \infty {\rm ~by~} (\ref{ch2th2.4.3}).
\end{align*}
Thus $[V, \lambda, M]_\infty^{G} (\Delta^{m-1}_{G}) \subset [V, \lambda, M]_\infty^{G} (\Delta^{m}_{G}).$ Hence  $[V, \lambda, M]_\infty^{G} (\Delta^{i}_{G}) \subset [V, \lambda, M]_\infty^{G} (\Delta^{m}_{G})$  for $i=1,2,....,m-1$.
\end{proof}
Now consider the following example to show the inclusion is strict.
\begin{example}
{\rm The sequence $x=(k^{m}) \in [V, \lambda, M]_\infty^{G} (\Delta^{m}_{G})$,  $M(x)=x$ $p_{k}=1$ $\forall$ $k \in \mathbb{N}$ and $\lambda_{n} =n$ for all $\eta \in \mathbb{N}$.\\
(If $x=(k^{m})$, then $\Delta^{m}_{G}x_{i}=(\ominus e)^m_{G} \odot m!_{G}$ and $\Delta^{m-1}_{G}x_{i}=(\ominus e)^{m+1}_{G} \odot m!_{G}(k \oplus (\frac{(m-1)}{2}G)$ for all $k \in \mathbb{N}.) \textrm{~ by  equation~} (\ref{eqem})$
Hence $x \notin [V, \lambda, M]_\infty^{G} (\Delta^{m-1}_{G})$.\\
This implies the above inclusion (\ref{thm2.4.3})  is strict.}
\end{example}

\begin{theorem} \label{ch2th2.4.4}
The sequence spaces $[V, \lambda, M,p]_{0}^{G}$ and  $[V, \lambda, M,p]_{\infty}^{G}$ are solid.
\end{theorem}
\begin{proof}
We give the proof for $[V, \lambda, M,p]_{0}^{G}$ . Let $(x_{i}) \in [V, \lambda, M,p]_{0}^{G}$ and $\alpha_{k}$ be any sequence of scalars such that $|\alpha_{k}|^{G} \leq 1$ for all $k \in \mathbb{N}$. Then we have 
\begin{align*}
 (e \oslash e^{\lambda_{\eta }}) \odot _G \sum _{k\in I_{n} }\left[  M \odot  \left(\frac{\left| \alpha_{k} \odot x_{i}  \right|^{G} }{\rho } G \right)\right]\\
 \leq (e \oslash e^{\lambda_{\eta }}) \odot _G \sum _{k\in I_{n} }\left[  M \odot  \left(\frac{\left| x_{i}  \right|^{G}}{\rho} G \right)\right] \rightarrow 1 \\{\rm~as~} n\rightarrow \infty.
\end{align*}
Hence $ (\alpha_{k} \odot x_{i}) \in [V, \lambda, M,p]_{0}^{G}$ for all sequences of scalars $(\alpha_{k})$ with $|\alpha_{k}| \leq 1$ for all $k \in \mathbb{N}$, whenever $(x_{i}) \in [V, \lambda, M,p]_{0}^{G}$.
\end{proof}
\begin{theorem}
For any Orlicz function $M$ which satisfies $\Delta_2-$ condition hence we have $[V, \lambda](\Delta^{m}_{G}) \subset [V, \lambda, M](\Delta^{m}_{G}) $ .
\end{theorem}
\begin{proof}
Let $x \in [V, \lambda](\Delta^{m}_{G})$ so that 
\begin{align*}
&  A_{n} \equiv (e \oslash e^{\lambda_{\eta }}) \odot {~}_G \sum _{k\in l_{n} }\left|\, \Delta _{G}^{m} x_{i} \ominus L\, \right| ^{G} \rightarrow 1,\\  {\rm ~} n\rightarrow \infty, {\rm ~for  some ~} L.
\end{align*}
Let $\varepsilon >0$ and choose $\delta$ with $1<\delta<e$ such that $M(t) < \varepsilon$ for $1 \leq t \leq \delta$. We can write 
\begin{align*}
& (e \oslash e^{\lambda_{\eta }}) \odot {~}_G \sum _{k\in l_{n} } M \odot \left(\left|\, \Delta _{G}^{m} x_{i} \ominus L\, \right|\right) ^{G}\\
&=  (e \oslash e^{\lambda_{\eta }}) \odot {~} _G \sum _{{\mathop{k\in I_{n} }\limits_{\left|\, \Delta _{G}^{m} x_{i}\ominus L\, \right|\, \ge \, \varepsilon }} }M \odot \left(\left|\, \Delta _{G}^{m} \left(x_{i} \right)\ominus L\, \right|^{G}\right)  \\
&\oplus  (e \oslash e^{\lambda_{\eta }}) \odot {~}_G \sum _{{\mathop{k\in I_{n} }\limits_{\left|\Delta _{G}^{m} x_{i}\ominus L \right| < \varepsilon }}} M \odot \left(\left| \Delta _{G}^{m} x_{i}   \ominus L \right|^{G} \right)\\
& < [K \odot \delta^{-1} \odot M(t)\odot  A_{n} ] \oplus\left[ (e \oslash e^{\lambda_{\eta }}) \odot (\lambda_{n} \varepsilon)\right]
\end{align*}
By taking $\eta \rightarrow \infty$, it follows that $x \in [V, \lambda, M](\Delta^{m}_{G})$.
\end{proof}
\begin{theorem}
{\rm For  $m$ be a positive integer ,  $M$ be any Orlicz function, \\
$[V, \lambda, M](\Delta^{m}_{G}) \subset S_{\lambda}(\Delta^{m}_{G})$.}
\end{theorem}
\begin{proof}
Let $x \in [V, \lambda, M](\Delta^{m}_{G})$ and $\varepsilon >1$ be given. Then 
\begin{align*}
& (e \oslash e^{\lambda_{\eta }}) \odot {~}_G \sum _{k\in l_{n} } \left[ M \odot \left( \frac{\left| \Delta^{m}_{G}x_{i} \ominus L\right|^{G}}{\rho}\right)\right]\\
& \geq (e \oslash e^{\lambda_{\eta }}) \odot {~} _G \sum _{{\mathop{k\in I_{n} }\limits_{\left|\Delta _{G}^{m} x_{i}\ominus L \right| \geq \varepsilon }}} \left[ M \odot \left( \frac{\left| \Delta^{m}_{G}x_{i} \ominus L\right|^{G}}{\rho}\right)\right]\\
& > (e \oslash e^{\lambda_{\eta }}) \odot M(\varepsilon/ \rho) \left| \{ k \in I_{n} : \left| \Delta^{m}_{G}x_{i} \ominus L\right| \geq \varepsilon \}\right|^{G}.
\end{align*}
Hence $x \in S_{\lambda}(\Delta^{m}_{G})$.
\end{proof}
\section{Dual Spaces}
Geometric $\alpha -$ dual and $\beta-$ dual of the sequence space $X$  in geometric calculus are denoted by ${X}^\alpha_{G}$ and ${X}^\beta_{G}$ and defined as follows:
\begin{align*}
&  X^{\alpha}_{G} = \{ a=(a_{k}) \in \omega(G)  : _{G}\sum_ {k=1} ^ {\infty } |a_{k} \odot x_{i}|< \infty \},\\
&   X^{\beta}_{G} = \{ a=(a_{k}) \in \omega(G) : _{G} \sum_ {k=1} ^ {\infty } a_{k}x_{i}  \textrm {~is convergent ~} \}
\end{align*}
 for each $ x \in X$ 
respectively.\\

If $p=(p_{n})$ is bounded, then 
\begin{align*}
V(\lambda, p) = \left\lbrace  x \in \omega : \sum _{n=0}^{\infty} \left( \frac{1}{\lambda_{n}} \sum _{k \in I_{n}} |x_{i}| \right)^{p_{n}} < \infty \right\rbrace
\end{align*}
 with the norm \cite{Bakery} 
$$||x|| =\inf \{ \lambda >0 : \rho\left(\frac{x}{\lambda} \right) \leq 1 \}$$
Now the above space in geometric difference sequence spaces in mth order  is given by
\begin{align*}
& V^\lambda_{p}(\Delta^{m}_{G}) = \left\lbrace  x=(x_{i})  \in \omega(G) : _{G}\sum _{\eta =1}^{\infty} \left(  (e \oslash e^{\lambda_{\eta }}) \odot _{G} \sum _{k \in I_{n}} \Delta^{m}_{G}x_{i} \right)^{\left(p_{k} \right)^{G} }< \infty \right\rbrace\\
& V^\lambda_{\infty}(\Delta^{m}_{G}) = \left\lbrace  x=(x_{i})  \in \omega(G) : \sup_ {\eta \in \mathbb{N}} \left| (e \oslash e^{\lambda_{\eta }}) \odot _{G} \sum _{k \in I_{n}}\Delta^{m}_{G}x_{i} \right|^{G}< \infty \right\rbrace
\end{align*}
Now for an operator $u:X \to X$ by $u(x)= (1,\cdots,1,x_{m+1},x_{m+2},\cdots)$ for all $x=(x_{i})\in X$. Consider the sets ${u}V^\lambda_{p}(\Delta^{m}_{G}) $ and ${u}V^\lambda_{\infty}(\Delta^{m}_{G})$ as 
\begin{align*}
{u}V^\lambda_{p}(\Delta^{m}_{G})=\left\lbrace x=(x_{i}): x \in V^\lambda_{p}(\Delta^{m}_{G}) \textrm {~~and~} x_1=x_2=\cdots=x_m=1\right\rbrace
\end{align*}
and 
\begin{align*}
{u}V^\lambda_{\infty}(\Delta^{m}_{G})=\left\lbrace x=(x_{i}): x \in V^\lambda_{\infty}(\Delta^{m}_{G}) \textrm {~~and~} x_1=x_2=\cdots=x_m=1\right\rbrace.
\end{align*}
\begin{lemma} \label{ch2l2.5.1}
If $x=(x_{i}) \in {u}V^\lambda_{\infty}(\Delta^{m}_{G})$, then $\sup_{k} e^{(\lambda_{k})^{-1}} \odot |\Delta^{m-1}_{G} x_{i}|^{G} < \infty$.
\end{lemma}
\begin{proof}
Let $x=(x_{i}) \in {u}V^\lambda_{\infty}(\Delta^{m}_{G})$, then 
\begin{equation} \label{2.5.1}
\sup_ {\eta \in \mathbb{N}} \left| (e \oslash e^{\lambda_{\eta }}) \odot _{G} \sum _{k \in I_{n}} \Delta^{m}_{G}x_{i} \right|^{G}< \infty.
\end{equation}
$I_{n}=[n-\lambda_{n}+1,n],{~~} n=1,2,3,\cdots$
Consider 
\begin{align*}
 _{G} \sum _{k \in I_{n}} \Delta^{m}_{G}x_{i}  &\leq {~~}_{G} \sum _{i=1} ^{n}\Delta^{m}_{G}(x_{i} \ominus x_{i+1})\\
&={~~}_{G} \sum _{i=1} ^{n}(\Delta^{m-1}_{G}(x_{i} \ominus \Delta^{m-1}_{G}x_{i+1})\\
&={~~} \Delta^{m-1}_{G}x_{1} \ominus \Delta^{m-1}_{G}x_{n+1}
\end{align*}
As $x_1=x_2=\cdots=x_m=1$, we get $\Delta^{m-1}_{G}x_{1}=1$. Therefore,\\
$ _{G} \sum _{k \in I_{n}} \Delta^{m}_{G}x_{i} = {~~} 1 \ominus \Delta^{m-1}_{G}x_{n+1}$.\\
Since $0_{G}$ is 1 hence we omit it and write 
$$ _{G} \sum _{k \in I_{n}} \Delta^{m}_{G}x_{i} = {~~}  \ominus \Delta^{m-1}_{G}x_{n+1}
 .$$
Now substituting this value in  (\ref{2.5.1}), \\
\begin{align*}
\sup_ {\eta \in \mathbb{N}} \left| (e \oslash e^{\lambda_{\eta }}) \odot ( \ominus \Delta^{m-1}_{G}x_{n+1})\right|^{G}< \infty.
\end{align*}
By general geometric arithmetic  property  
\begin{align*}
& \leq  \sup_ {\eta \in \mathbb{N}} \left| (e \oslash e^{\lambda_{\eta }}) \right|^{G}\odot  \left|( \ominus \Delta^{m-1}_{G}x_{n+1})\right|^{G}< \infty\\
& = \sup_ {\eta \in \mathbb{N}} \left| (e \oslash e^{\lambda_{\eta }}) \right|^{G}\odot  \left| \Delta^{m-1}_{G}x_{n+1}\right|^{G}< \infty.
\end{align*}
Since $(e \oslash e^{\lambda_{\eta }})=e^{(\lambda_{n})^{-1}} >1$, so 
\begin{equation}\label{eq2.5.2}
\sup_ {\eta \in \mathbb{N}}  e^{(\lambda_{n})^{-1}} \odot  \left| \Delta^{m-1}_{G}x_{n+1}\right|^{G}< \infty.
\end{equation}
Now,
\begin{align*}
& \sup_ {\eta \in \mathbb{N}}  e^{(\lambda_{n})^{-1}} \odot  \left| \Delta^{m-1}_{G}x_{n}\right|^{G}\\
&  =\sup_ {\eta \in \mathbb{N}}  e^{(1-\frac{1}{\lambda_{n}})(\lambda_{n}-1)^{-1}} \odot  \left| \Delta^{m-1}_{G}x_{n}\right|^{G}\\
&  \leq \sup_ {\eta \in \mathbb{N}}  e^{(\lambda_{n}-1)^{-1}} \odot  \left| \Delta^{m-1}_{G}x_{n}\right|^{G}.
\end{align*}
Replacing the R.H.S. $n=n+1$ and comparing it with equation(\ref{eq2.5.2}) we get
$$ \sup_ {\eta \in \mathbb{N}}  e^{(\lambda_{n})^{-1}} \odot  \left| \Delta^{m-1}_{G}x_{n}\right|^{G} < \infty $$
Replacing $n$ as $k$ we have 
$$\sup_{k} e^{(\lambda_{k})^{-1}} \odot |\Delta^{m-1}_{G} x_{i}|^{G} < \infty . $$
\end{proof}
\begin{corollary}\label{ch2cl2.5.1}
{\rm If $$\sup_{k} e^{{k}^{-1} }\odot |\Delta^{m-1}_{G} x_{i}|^{G} < \infty,$$ then $$\sup_{k} e^{(k)^{-m}} \odot | x_{i}|^{G} < \infty . $$}
\end{corollary}
\begin{lemma}
{\rm   A sequence $x=(x_{i}) \in {u}V^\lambda_{\infty}(\Delta^{m}_{G})$, then it gives
$$\sup_{k} e^{(\lambda_{k})^{-m}} \odot | x_{i}|^{G} < \infty .$$}
\end{lemma}
\begin{proof}
Let us consider $x=(x_{i}) \in {u}V^\lambda_{\infty}(\Delta^{m}_{G})$ , then from Lemma(\ref{ch2l2.5.1})
$$\sup_{k} e^{(\lambda_{k})^{-1}} \odot |\Delta^{m-1}_{G} x_{i}|^{G} < \infty,$$ which then by corollary (\ref {ch2cl2.5.1}) gives that 
$$\sup_{k} e^{(\lambda_{k})^{-m}} \odot | x_{i}|^{G} < \infty.$$
\end{proof}
\begin{theorem}
$$ [uV^\lambda_{\infty}(\Delta^{m}_{G})]^{\alpha} =\left\lbrace a=(a_{k}): {~}_{G} \sum_{k=1}^{\infty} e^{(\lambda_k)^{m}} \odot | a_{k}|^{G} < \infty \right\rbrace.$$
\end{theorem}
\begin{proof}
We consider $$W=\left\lbrace a=(a_{k}): {~}_{G} \sum_{k=1}^{\infty} e^{(\lambda_k)^{m}} \odot | a_{k}|^{G} < \infty \right\rbrace.$$ Let $a=(a_{k}) \in W$; then for any $x=(x_{i}) \in uV^\lambda_{\infty}(\Delta^{m}_{G}),$ we have 
\begin{align*}
& _{G} \sum_{k=1}^{\infty}| a_{k}\odot x_{i}|^{G}\\
& =_{G} \sum_{k=1}^{\infty}e^{(\lambda_k)^{m}} \odot | a_{k}|^{G}\odot (e^{(\lambda_k)^{-m}} \odot |x_{i}|^{G})\\
& \leq \sup_{k \in \mathbb{N}}(e^{(\lambda_k)^{-m}}|x_{i}|^{G}) \odot _{G} \sum_{k=1}^{\infty}e^{(\lambda_k)^{m}} \odot | a_{k}|^{G}
\end{align*}
By Lemma(\ref{ch2l2.5.1}) this is finite . Also $x=(x_{i})$ is arbitrary, we have 
$a=(a_{k}) \in  [uV^\lambda_{\infty}(\Delta^{m}_{G})]^{\alpha} $. \\ Hence 
\begin{equation}\label{ch22.5.3}
W \subseteq [uV^\lambda_{\infty}(\Delta^{m}_{G})]^{\alpha} .
\end{equation}

Conversely, let $a \in [uV^\lambda_{\infty}(\Delta^{m}_{G})]^{\alpha}.$ \\
$_{G} \sum_{k=1}^{\infty}| a_{k}\odot x_{i}|^{G} $ is less than infinite for every $x=(x_{i}) \in uV^\lambda_{\infty}(\Delta^{m}_{G}). $ Now, look about the sequence $x=(x_{i})$ that is defined by 
\begin{align}
 x_{i}=\left\{\begin{array}{lr}
 1,&(k\leq m),\\e^{(\lambda_k)^{m}},&(k>m),
 \end{array}\right. 
 \end{align}
 This is $\in  uV^\lambda_{\infty}(\Delta^{m}_{G})$. Hence, 
 \begin{align*}
_{G}\sum_{k=m+1}^{\infty}  | a_{k} \odot (e^{(\lambda_k)^{m}} |^{G} < \infty.
 \end{align*}
 Now,
 \begin{align*}
& _{G}\sum_{k=1}^{\infty}  | a_{k} \odot (e^{(\lambda_k)^{m}} |^{G}\\
 & = _{G}\sum_{k=1}^{m}  | a_{k} \odot (e^{(\lambda_k)^{m}} |^{G} \oplus  _{G}\sum_{k=m+1}^{\infty}  | a_{k} \odot (e^{(\lambda_k)^{m}} |^{G}
 < \infty
 \end{align*}
 Hence $a \in W$. Thus,
 \begin{equation}\label{ch22.5.5}
 [uV^\lambda_{\infty}(\Delta^{m}_{G})]^{\alpha} \subseteq W.
\end{equation}
From (\ref{ch22.5.3}) and (\ref{ch22.5.5})
\begin{align*}
W =  [uV^\lambda_{\infty}(\Delta^{m}_{G})]^{\alpha}
\end{align*}
\end{proof}
\begin{theorem}
 $ [V^\lambda_{\infty}(\Delta^{m}_{G})]^{\alpha}= [uV^\lambda_{\infty}(\Delta^{m}_{G})]^{\alpha}$
\end{theorem}
\begin{proof}
$ [uV^\lambda_{\infty}(\Delta^{m}_{G})]^{\alpha} \subset [V^\lambda_{\infty}(\Delta^{m}_{G})]^{\alpha}$, so we have 
\begin{equation}\label{ch22.5.6}
[V^\lambda_{\infty}(\Delta^{m}_{G})]^{\alpha}\subseteq  [uV^\lambda_{\infty}(\Delta^{m}_{G})]^{\alpha}
\end{equation} 
Conversely  let $a=(a_{k}) \in [uV^\lambda_{\infty}(\Delta^{m}_{G})]^{\alpha}$,  $$_{G} \sum_{k=1}^{\infty}| a_{k}\odot x_{i}|^{G} is less than infinite $$ for each $x=(x_{i}) \in uV^\lambda_{\infty}(\Delta^{m}_{G})$. \\
Now take any sequence $x^{\prime}=(x_{i} ^{\prime}) \in V^\lambda_{\infty}(\Delta^{m}_{G})$,  the following sequence $(1,1,1, \cdots,1,x^{\prime}_{m+1},x^{\prime}_{m+2}, \cdots) \in  [uV^\lambda_{\infty}(\Delta^{m}_{G})]^{\alpha}$ and 
 \begin{align*}
_{G}\sum_{k=m+1}^{\infty}  | a_{k} \odot x_{i}^{\prime} |^{G} < \infty.
 \end{align*}
 Now,
 \begin{align*}
& _{G}\sum_{k=1}^{\infty}  | a_{k} \odot x_{i}^{\prime} |^{G}\\
 & = _{G}\sum_{k=1}^{m}  | a_{k} \odot x_{i}^{\prime} |^{G} \oplus  _{G}\sum_{k=m+1}^{\infty}  | a_{k} \odot x_{i}^{\prime} |^{G}
 < \infty
 \end{align*}
 $\forall$ $x=(x_{i}^{\prime}) \in V^\lambda_{\infty}(\Delta^{m}_{G})$. Therefore, the sequence $a=(a_{k}) \in  [V^\lambda_{\infty}(\Delta^{m}_{G})]^{\alpha}$ and hence 
\begin{equation}\label{ch22.5.7}
[uV^\lambda_{\infty}(\Delta^{m}_{G})]^{\alpha} \subseteq [V^\lambda_{\infty}(\Delta^{m}_{G})]^{\alpha} 
\end{equation} 
Hence from equations (\ref{ch22.5.6}) and (\ref{ch22.5.7}) we obtain our required result.
\end{proof}

\end{document}